\documentclass{amsart}
\usepackage{amsfonts}
\usepackage{amsmath}
\usepackage{amsthm}
\usepackage{amsfonts,amssymb}
\usepackage{CJK}
\usepackage{enumerate}

\newtheorem{thm}{Theorem}[section]

\newtheorem{lem}[thm]{Lemma}

\theoremstyle{remark}
\newtheorem{rem}[thm]{Remark}

\numberwithin{equation}{section}

\newcommand{\al}{\alpha}

\def\oz{\omega}
\def\lz{\lambda}
\def\Lz{\Lambda}
\def\Oz{\Omega}

\def\dz{\delta}
\def\az{\alpha}
\def\gz{\gamma}

\def\tz{\theta}

\def\({\Bigl(}
\def \){ \Bigr)}

\def\sub{\substack}

 \def\az{{\alpha}}
 
 \def\gz{{\gamma}}
 
 \def\tz{{\theta}}
 \def\lz{{\lambda}}
 \def\dz{{\delta}}
 \def\oz{{\omega}}

 \def\RR{{\mathbb R}}

\def\va{\varepsilon}

\def\Lz{\Lambda}

\def\va{\varepsilon}

\def\Lz{\Lambda}

\begin{document}
\def\RR{\mathbb{R}}
\def\Exp{\text{Exp}}
\def\FF{\mathcal{F}_\al}

\title[] {Optimal randomized quadrature for weighted Sobolev and Besov classes with the Jacobi weight on the ball}
\author[]{Jiansong Li} \address{ School of Mathematical Sciences, Capital Normal
University, Beijing 100048,
 China}
\email{2210501007@cnu.edu.cn}

\author[]{ Heping Wang} \address{ School of Mathematical Sciences, Capital Normal
University, Beijing 100048,
 China}
\email{ wanghp@cnu.edu.cn}

\keywords{ Optimal quadrature error; Weighted Sobolev classes;
Weighted Besov classes;  Deterministic case setting; Randomized
case setting; Filtered hyperinterpolation. }

\subjclass[2010]{41A55, 65D30, 65D32}

\begin{abstract}

  We consider the numerical integration $${\rm INT}_d(f)=\int_{\mathbb{B}^{d}}f(x)w_\mu(x)dx $$ for the weighted Sobolev classes
$BW^{r}_{p,\mu}$ and the weighted Besov classes
$BB_\tau^r(L_{p,\mu})$ in the randomized case setting, where
$w_\mu, \,\mu\ge0,$ is the classical Jacobi weight on the ball
$\Bbb B^d$, $1\le p\le \infty$, $r>(d+2\mu)/p$, and
$0<\tau\le\infty$.
 For the above two classes, we obtain the orders of the optimal quadrature errors in the randomized case setting are $n^{-r/d-1/2+(1/p-1/2)_+}$.
 Compared to the orders $n^{-r/d}$ of the optimal quadrature errors in the deterministic case setting, randomness can effectively improve the order of convergence when $p>1$.
\end{abstract}

\maketitle
\input amssym.def

\section{Introduction}

\

Let $F_d$ be a class of continuous functions on $D_d$, where $D_d$ is a
compact subset of the Euclidean space $\mathbb{R}^{d}$ with a probability measure $\rho$. The integral of a continuous function $f:\,F_d\rightarrow\,\Bbb R$ denotes by
\begin{equation}\label{1.1}{\rm INT}_d(f)=\int_{D_d}f(x)d\rho(x).\end{equation}

We want to approximate this integral ${\rm INT}_d(f)$ by
(deterministic) algorithms of the form
$$A_n(f):=\varphi_n(f(x_1), f(x_2),\dots,f(x_n)),$$where $x_j\in D_d$ can be chosen adaptively and $\varphi_n:\Bbb R^n\rightarrow\Bbb R$
is an arbitrary mapping. Adaption means that the selection of
$x_j$ may depend on the already computed values $f(x_1),
f(x_2),\dots,f(x_{j-1})$.  We denoted by $\mathcal{A}_n^{\rm det}$
the class of  all algorithms of this form. If $x_1,\dots,x_n$ are
fixed and $\varphi_n$ is linear, i.e.,
$$A_n(f)=\sum_{j=1}^n\lz_jf(x_j),\ \ \lz_j\in\Bbb R,\
j=1,\dots,n,$$ then the algorithm $A_n$ is called a linear
algorithm. Such linear algorithm $A_n$ is also called a quadrature
formula. We say that a quadrature formula $A_n$ is positive if
$\lz_j>0, \, j=1,\dots,n$.

The deterministic case error of $A_n$ on $F_d$ is given by
\begin{equation*}
  e^{\rm det}(F_d,A_n):=\sup\limits_{f\in F_d}|{\rm INT}_d(f)-A_n(f)|,
\end{equation*}and the minimal (optimal) deterministic case error on $F_d$ given by
\begin{equation*}
  e_n^{\rm det}(F_d):=\inf\limits_{A_n\in\mathcal{A}_n^{\rm det}}e^{\rm det}(F_d,A_n).
\end{equation*}

It was well known (see \cite{B3}) that if  $F_d$ is convex and
balanced, then $e_n^{\rm det}(F_d)$ can be achieved by linear
algorithms. Hence $e_n^{\rm det}(F_d)$  is also called the
optimal quadrature error.

Randomized algorithms, called also Monte-Carlo algorithms, are
understood as $\Sigma\otimes \mathcal B( F_d)$ measurable
functions
$$(A^\oz)=(A^\oz(\cdot))_{\oz\in\Omega}: \Omega\times  F_d\to
\Bbb R,$$ where $\mathcal B( F_d)$  denotes Borel $\sigma$-algebra
of $ F_d$, $(\Omega, \Sigma, \mathcal P)$ is a suitable
probability space, and for any fixed $\oz\in \Omega$, $A^\oz$ is a
deterministic method with cardinality $n(f,\oz)$. The number
$n(f,\oz)$ may be randomized and adaptively depend on the input,
and the cardinality of $(A^\oz)$ is then  defined by $${\rm
Card}(A^\oz):=\sup_{f\in F_d}\mathbb E_\oz\,
n(f,\oz):=\sup\limits_{f\in F_d}\int_{\Omega} n(f,\omega)d\mathcal
P(\omega).$$ We denote by $\mathcal{A}_n^{\rm ran}$  the class of
all randomized algorithms with cardinality not exceeding $n$.

The randomized case error of $(A^\omega)$ on $F_d$ is defined by
$$e^{\rm ran}(F_d,(A^\omega)):=\sup\limits_{f\in F_d}\Bbb E_\omega|{\rm INT}_d(f)-A^{\omega}(f)|,$$and the minimal  (optimal)  randomized case error on $F_d$ is defined by
\begin{equation*}
  e_n^{\rm ran}(F_d):=\inf\limits_{(A^\omega)\in\mathcal{A}_n^{\rm ran}}e^{\rm ran}(F_d,(A^\omega)).
\end{equation*}

There are many papers devoted to investigating the integration
problem \eqref{1.1} in the deterministic  and randomized case
settings. Compared to deterministic algorithms, randomized
algorithms may  speed up the order of convergence in many cases,
especially for integration problem. We recall some known results.

 Throughout the paper, the notation
$a_{n}\asymp b_{n}$ means $a_{n}\lesssim b_{n}$  and $a_{n}\gtrsim
b_{n}$. Here, $a_{n}\lesssim b_{n}\,(a_{n}\gtrsim b_{n})$ means
that there exists a constant $c>0$ independent of $n$ such that
$a_{n}\leq c b_{n}\,(b_{n}\leq c a_{n})$.

\vskip 1mm

(1) Consider the classical Sobolev class $BW_{p}^{r}([0,1]^{d})$,
$1\le p\le\infty$, $r\in\Bbb N$, defined by
 $$BW_{p}^{r}([0,1]^{d})=\big\{f\in L_{p}([0,1]^{d})\,\big|\, \sum_{|\alpha|_{1}\le r}\|D^{\alpha}f\|_{p}\leq1\big\},$$
and  the H\"older class $C^{k,\gz}_d$, $k\in\Bbb N_0$,
$0<\gz\le1$, defined by
$$C^{k,\gamma}_d:=\big\{f\in C([0,1]^d)\,\big|\,|D^\az f(x)-D^\az f(y)|\le \max\limits_{1\le i\le d}|x_i-y_i|^\gz, \, |\az|_1= k\big\},$$
where $\az\in \Bbb N_0^d,\
|\alpha|_{1}:=\sum\limits_{i=1}^d\alpha_i$, and $D^\az f$ is the
partial derivative  of order $\az$ of $f$ in the sense of
distribution. Bakhvalov in \cite{B1} and \cite{B2} proved that
 \begin{equation*}e_{n}^{\rm det}(C^{k,\gz}_d)\asymp n^{-\frac{k+\gz}{d}}\ \ \ {\rm and}\ \ \  e_{n}^{\rm det}(BW_{\infty}^{r}([0,1]^{d}))\asymp n^{-\frac{r}{d}}.\end{equation*}
 Novak extended the second equivalence  result in \cite{N2} and \cite{N3}, and proved that for $1\leq p<\infty$  and $r>d/p$,
 \begin{equation*}e_{n}^{\rm det}(BW_{p}^{r}([0,1]^{d}))\asymp n^{-\frac{r}{d}}. \end{equation*}
 Meanwhile, Novak considered the randomized case errors of the above two classes in \cite{N2} and \cite{N3},  and proved that
 \begin{equation*}e_{n}^{\rm ran}(C^{k,\az}_d)\asymp
n^{-\frac{k+\az}{d}-\frac{1}{2}},\end{equation*}
and for $1\leq p\le\infty$ and $r>d/p$,
\begin{equation*}e_{n}^{\rm ran}(BW_{p}^{r}([0,1]^{d}))\asymp n^{-\frac{r}{d}-\frac{1}{2}+(\frac{1}{p}-\frac{1}{2})_{+}},\end{equation*}where
$a_{+}=\max(a,0)$.

\vskip 1mm

 (2) Consider the anisotropic Sobolev class
$BW_{p}^{{\bf r}}([0,1]^{d}),\ 1\leq p\leq\infty,\ {\bf
r}=(r_{1},\cdots, r_{d})\in \mathbb{N}^{d}$, defined by
$$ BW_{p}^{{\bf r}}([0,1]^{d})=\big\{f\in
L_{p}([0,1]^{d})\,\big|\,\sum_{j=1}^{d}\big\|\frac{\partial^{r_{j}}f}{\partial
x_{j}^{r_{j}}}\big\|_p\le 1\big\}.$$  Fang and Ye in \cite{FY1}
obtained for $g({\bf r})>d/{p}$,
\begin{equation*}e_{n}^{\rm det}(BW_{p}^{{\bf r}}([0,1]^{d}))\asymp n^{-g({\bf r})},\end{equation*} and for $g({\bf
r})>1/{p}$,
\begin{equation*}e_{n}^{\rm ran}(BW_{p}^{{\bf r}}([0,1]^{d}))\asymp n^{-g({\bf r})-\frac{1}{2}+(\frac{1}{p}-\frac{1}{2})_{+}},
\end{equation*}where $g({\bf
r})=\big(\sum_{j=1}^{d}r_{j}^{-1}\big)^{-1}$. For the anisotropic
H\"older-Nikolskii  classes, Fang and Ye  obtained the similar
results in \cite{FY1}.

\vskip 1mm

(3) Consider the Sobolev class with bounded mixed derivative
$BW_{p}^{r,\rm mix}([0,1]^{d}),\ r\in\mathbb {N}, \ 1\leq
p\leq\infty$, defined by
$$BW_{p}^{r,\rm mix}([0,1]^{d})=\big\{f\in
L_{p}([0,1]^{d})\,\big|\,\sum_{|\alpha|_{\infty}\leq
r}\|D^{\alpha}f\|_{p}\leq1\big\},$$where
$|\alpha|_{\infty}:=\max\limits_{1\le i\le d}\az_i$. The authors
in \cite{BV, FKK,  SM, T3} obtained for $r>1/{p}$  and  $
1<p<\infty$,
\begin{equation*}e_{n}^{\rm det}(BW_{p}^{r,\rm mix}([0,1]^{d}))\asymp n^{-r}(\log n)^{\frac{d-1}{2}} .\end{equation*}
It was shown in \cite{KN, NUU, U} that for $r>\max\{1/p,1/2\}$ and
$1< p<\infty$,
\begin{equation*}e_{n}^{\rm ran}(BW_{p}^{r,\rm mix}([0,1]^{d}))\asymp n^{-r-\frac{1}{2}+(\frac{1}{p}-\frac{1}{2})_+}.\end{equation*}

\vskip 1mm

(4) For the Sobolev class $BW_p^r(\Bbb S^{d-1})$, $1\le p\le
\infty$, $r>0$, on the sphere $\Bbb S^{d-1}$,    it was proved in
\cite{BH, H, Wang} that for  $r>(d-1)/p$,
\begin{equation*}e_{n}^{\rm det}(BW_{p}^{r}(\mathbb{S}^{d-1}))\asymp n^{-\frac{r}{d-1}}.\end{equation*}
Wang and Zhang in \cite{WZ} obtained  for $r>(d-1)/p$,
\begin{equation*}e_{n}^{\rm ran}(BW_{p}^{r}(\mathbb{S}^{d-1}))\asymp  n^{-\frac r{d-1}-\frac{1}{2}+(\frac{1}{p}-\frac{1}{2})_+}.
\end{equation*}

\vskip 1mm

(5) For the generalized Besov class $BB_{p,\tz}^\Omega(\Bbb
S^{d-1})$, $1\le p,\tz\le \infty,$ with the smoothness index
$\Omega $ satisfying some conditions,  Duan and Ye in \cite{DY}
obtained
\begin{equation*}e_{n}^{\rm det}(BB_{p,\tz}^\Omega(\Bbb S^{d-1}))\asymp \Omega(n^{-\frac1{d-1}}),\end{equation*}
and
\begin{equation*}e_{n}^{\rm ran}(BB_{p,\tz}^\Omega(\Bbb S^{d-1}))\asymp \Omega(n^{-\frac1{d-1}}) n^{-\frac{1}{2}+(\frac{1}{p}-\frac{1}{2})_+}.
\end{equation*}
We remark that if $\Oz(t)=t^r$, then $BB_{p,\tz}^\Omega(\Bbb
S^{d-1})$ recedes to the usual Besov class $BB^r_\tz(L_p(\Bbb
S^{d-1}))$.

\vskip 1mm

(6) Dai and Wang in \cite{DW} investigated  the weighted Besov
class $BB_\tau^r(L_{p,w}(\Bbb S^{d-1}))$, $r>0,\, 0<\tau\le
\infty,\, 1\le p\le \infty,$ with an $A_\infty$ weight $w$ on
$\Bbb S^{d-1}$. They obtained  for $r>s_w/p$,
\begin{equation*}
e_{n}^{\rm det}(BB_\tau^r(L_{p,w}(\Bbb S^{d-1})))\asymp
n^{-\frac{r}{d-1}},
\end{equation*}
where $s_w$ is a critical index for the $A_\infty$ weight $w$.
This generalized the unweighted result of \cite{HMS}. Meanwhile,
they also obtained the corresponding results for the weighted
Besov classes on the unit ball and on the standard simplex of the
Euclidean space $\Bbb R^d$.

\vskip 1mm

The above results  indicate that randomized algorithms effectively
improve the optimal rate of convergence in many cases. There is a
vast literature of integration problems in the deterministic and
randomized case settings, see for example, \cite{DW, Du1, Du2, DU,
FY1, H4, N3, N4, UU, W1, Y}. However, as far as we know, there are
few results about integration problem on the unite ball $\Bbb B^d$
in the randomized case setting.

\vskip 2mm

Let  $\mathbb{B}^{d}=\big\{x\in \mathbb{R}^d\,|\,|x|\leq 1\big\}$
be the unit ball of $\mathbb{R}^d$, where $x\cdot y$ is the usual
inner product and $|x|=(x\cdot x)^{1/2}$ is the usual Euclidean
norm. We denote by $L_{p,\mu}\equiv L_{p,\mu}(\Bbb B^{d}),\,
0<p<\infty,$ the space of  all measurable functions with finite
quasi-norm
$$\|f\|_{p,\mu}:=\Big(\int_{\mathbb{B}^{d}}|f(x)|^pw_\mu(x)dx\Big)^{1/p},$$
where $w_\mu(x)=b_{d}^\mu(1-|x|^2)^{\mu-1/2},\ \mu\ge0$ is the
classical Jacobi weight on $\Bbb B^{d}$, normalized by $\int_{\Bbb
B^{d}}w_\mu(x)dx=1$. When $p=\infty$ we consider the space  of
continuous functions $C(\Bbb B^{d})$ with the uniform norm. Let
$BW_{p,\mu}^{r}$ and $BB_\tau^r(L_{p,\mu})$,  $1\leq p\leq
\infty$, $r>0$, $0<\tau\le \infty$,  denote the weighted Sobolev
class and the weighted Besov class on $\mathbb{B}^{d}$,
respectively (see the precise definitions in Section 2). We remark
that  if $r>(d+2\mu)/p$, then the spaces $W_{p,\mu}^r$ and
$B_\tau^r(L_{p,\mu})$ are compactly embedded into the space of
continuous functions $C(\mathbb{B}^{d})$.

This paper is concerned with numerical integration on $\Bbb B^d$
\begin{equation}\label{1.2} {\rm
INT}_d(f)=\int_{\mathbb{B}^{d}}f(x)w_\mu(x)dx.
\end{equation}

 For the
weighted Besov class $BB_\tau^r(L_{p,\mu})$, $1\le p\le \infty$,
$0<\tau\le \infty$,  and $r>(d+2\mu)/p$, it follows from \cite{DW}
that
\begin{equation}\label{1.3}
e_{n}^{\rm det}(BB_\tau^r(L_{p,\mu}))\asymp n^{-\frac{r}{d}}.
\end{equation}

For the weighted Sobolev class $BW_{p,\mu}^r$, $1\le p\le \infty$,
$r>(d+2\mu)/p$, we obtain the similar result as follows.

\begin{thm}\label{thm1}Let $1\le p\le\infty$ and
$r>(d+2\mu)/{p}$. Then we have
\begin{equation}\label{1.4}e_{n}^{\rm det}(BW_{p,\mu}^{r})\asymp
n^{-\frac{r}{d}}.
\end{equation}\end{thm}

From \eqref{1.3} and \eqref{1.4}, we know that the integration
problems \eqref{1.2} for the weighted Besov class
$BB_\tau^r(L_{p,\mu})$ and  the weighted Sobolev class
$BW_{p,\mu}^r$ is ``intractable" in the deterministic setting if
$d$ is much larger than $r$. So it is natural to ask whether
randomness improves the order of convergence. In this paper we
investigate randomized quadrature for $BW_{p,\mu}^{r}$ and
$BB_\tau^r(L_{p,\mu})$. We obtain their sharp asymptotic orders of
quadrature errors in the randomized case setting, and find that
randomized algorithms provide a faster rate than that of
deterministic ones for $p>1$. Our main results can formulated as
follows.

\begin{thm}\label{thm2}
Let  $1\le p\le\infty$, $0<\tau\le\infty$, and $r>(d+2\mu)/{p}$.
Then we have
\begin{equation}\label{1.5}e_{n}^{\rm ran}(BX_p^{r})\asymp
n^{-\frac{r}{d}-\frac{1}{2}+(\frac{1}{p}-\frac{1}{2})_+},
\end{equation}where $X_p^r$ denotes $W_{p,\mu}^{r}$ or $B_\tau^r(L_{p,\mu})$.
\end{thm}

\begin{rem}\label{rem2} We compare the results in the deterministic and randomized case settings. For $p=1$, the order of convergence is the same,
which means that  randomness does not help for $p=1$. Randomness
does help for $1<p\leq\infty.$ Indeed, randomness improve the
order of convergence by a factor $n^{1-1/p}$ for $1<p<2$ and
$n^{1/2}$ for $2\leq p\leq\infty$.
\end{rem}

The organization of the paper is the following. Section \ref{sec2}
presents some facts about harmonic analysis on the ball. In
Section 3 we use the filtered hyperinterpolation operators to
approximate functions in  $W_{p,\mu}^{r}$ or
$B_\tau^r(L_{p,\mu})$, and show that  the filtered
hyperinterpolation operators are asymptotically optimal algorithms
in the sense of optimal recovery in some cases.  We also give the
proof of Theorem \ref{thm1}. Section \ref{sec3} and Section
\ref{sec4} are devoted to proving the upper and lower estimates of
the quantities $e_{n}^{\rm ran}(BX_p^{r})$ as in Theorem
\ref{thm2}, respectively.

\section{Preliminaries}\label{sec2}

This section is devoted to give some basic knowledge about harmonic analysis on the unit ball $\Bbb B^d$.

For the classical Jacobi weight
$$w_\mu(x)=b_{d}^\mu(1-|x|^2)^{\mu-1/2},\ \mu\ge0,\
b_{d}^\mu=\Big(\int_{\Bbb B^d}(1-|x|^2)^{\mu-1/2}dx\Big)^{-1},\
$$ on $\mathbb{B}^{d}$, denote by $L_{p,\mu}\equiv L_{p,\mu}(\Bbb
B^d)\,(0< p< \infty)$
 the space of all Lebesgue measurable functions $f$ on
$\mathbb{B}^{d}$ with the finite quasi-norm
$$\|f\|_{p,\mu}:=\Big(\int_{\mathbb{B}^{d}}|f(x)|^pw_\mu(x)dx\Big)^{1/p}.$$And when $p=\infty$ we consider the space  of
continuous functions $C(\Bbb B^d)$ with the uniform norm. In particular, $L_{2,\mu}$ is a Hilbert space with inner product $$\langle f,g\rangle_\mu:=\int_{\Bbb B^d}f(x)g(x)w_\mu(x)dx,\ {\rm for}\ f,g\in L_{2,\mu}.$$

Let $\Pi_n^{d}$ be the space of all polynomials in $d$ variables
of total degree at most $n$.  We denote  by
$\mathcal{V}_n^d(w_\mu)$ the space of all  polynomials of degree
$n$ which are orthogonal to lower degree polynomials in
$L_{2,\mu}.$ It is well known (see \cite[p.38 or p.229]{DX}) that
the spaces $\mathcal{V}_n^d(w_\mu)$ are just the eigenspaces
corresponding to the eigenvalues $-n(n+2\mu+d-1)$ of the
second-order differential operator
$$
 D_\mu:=\triangle-(x\cdot\nabla)^2-(2\mu+d-1)x\cdot \nabla,
$$
where $\triangle$ and $\nabla$ are the Laplace operator and gradient operator, respectively. More precisely, $$D_\mu P=-n(n+2\mu+d-1)P,\ \ {\rm for\ all}\ P\in\mathcal{V}_n^d(w_\mu).$$

It is easy to see that the spaces $\mathcal{V}_n^d(w_\mu)$  are
mutually orthogonal in $L_{2,\mu}$. Let
$\{\phi_{nk}\equiv\phi_{nk}^d:\, k=1,2,\dots,a_n^d\}$ be a fixed
orthonormal basis for $\mathcal{V}_n^d(w_\mu)$, where $a_n^d:={\rm
dim}\,\mathcal{V}_n^d(w_\mu)$. Then $$\{\phi_{nk}:\,
k=1,2,\dots,a_n^d,\, n=0,1,2,\dots\}$$ is an orthonormal basis for
$L_{2,\mu}$.
 The orthogonal projector ${\rm Proj}_n :\, L_{2,\mu}\rightarrow \mathcal{V}_n^d(w_\mu)$ can be written
 as\begin{align*}
     ({\rm Proj}_n f)(x)&=\sum\limits_{k=1}^{a_n^d}\langle f,\phi_{nk}\rangle\phi_{nk}(x)
     =\langle f,P_n(w_\mu; x,\cdot)\rangle_\mu,
   \end{align*}
 where $P_n(w_\mu; x, y)=\sum\limits_{k=1}^{a_n^d}\phi_{nk}(x)\phi_{nk}(y)$ is the reproducing kernel of $\mathcal{V}_n^d(w_\mu)$. See \cite{Xu1} for more details about $P_n(w_\mu; x, y)$.

 For $r>0$, we define the fractional power $(-D_\mu)^{r/2}$ of the operator $-D_\mu$ on $f$ by
$$(-D_\mu)^{r/2}f:=\sum\limits_{k=0}^\infty (k(k+2\mu+d-1))^{r/2}{\rm Proj}_kf,$$in the sense of distribution.
By \cite{Xu}, we have for any $P\in\Pi_n^d$,
\begin{equation}\label{122}
 \|(-D_\mu)^{r/2}P\|_{p,\mu}\lesssim n^r\|P\|_{p,\mu}.
\end{equation}

 Given $r>0$ and $1\le p\le\infty$, we define the weighted Sobolev space by
$$W_{p,\mu}^r\equiv W_{p,\mu}^r(\Bbb B^d):=\big\{f\in L_{p,\mu}\ \big|\ \|f\|_{W_{p,\mu}^r}:=\|f\|_{p,\mu}+\|(-D_\mu)^{r/2}f\|_{p,\mu}<\infty\big\},$$
while the weighted Sobolev class $BW_{p,\mu}^r$ is defined to be
the unit ball of the weighted Sobolev space  $W_{p,\mu}^r$. We
remark that if $r>(d+2\mu)/p$, then $W_{p,\mu}^r$ is compactly
embedded into $C(\Bbb B^d)$.

Let $\eta \in C^\infty[0,+\infty)$ (a ``$C^\infty$-filter'')
satisfy
$$\chi_{\left[0,1\right ]}\leq\eta\leq\chi_{\left[0,2\right]}.$$Here, $\chi_A$ denotes the characteristic function of $A$ for
$A\subset \Bbb R$. For $L\in\Bbb N$, we define the filtered
polynomial operator by
\begin{equation}\label{3.1}
V_L(f)(x)\equiv V_{L,\eta}(f)(x):= \sum_{k=0}^{\infty}\eta (\frac
kL){\rm Proj}_{k}(f)(x)=\langle f, K_{L,\eta}(x,\cdot)\rangle_\mu,
\end{equation}where $f\in L_{1,\mu}$, and \begin{equation}\label{98-0}K_{L,\eta}(x,y)=\sum_{k=0}^\infty \eta(\frac kL)P_k(w_\mu; x,y),\ x,y\in\Bbb B^d.\end{equation}
Then the following properties hold (see, for example, \cite{PX})):

(a) $V_L(f)\in \Pi^d_{2L-1}$ for any $f\in L_{1,\mu}$;

(b) $P=V_L(P)$ for any  $P\in\Pi_L^d$;

(c)
$\|V_L\|:=\|V_L\|_{(\infty,\infty)}=\|V_L\|_{(1,1)}=\sup\limits_{x\in\mathbb{B}^d}\|K_{L,\eta}(x,\cdot)\|_{1,\mu}\lesssim1$;

(d) $\|V_L\|_{(p,p)}\le \|V_L\|\lesssim 1$ for $1\le p\le \infty$;

(e) $\|f-V_L(f)\|_{p,\mu}\le
(1+\|V_L\|_{(p,p)})E_L(f)_{p,\mu}\lesssim E_L(f)_{p,\mu}$ for
$1\le p\le \infty$,

\noindent where $$\|A\|_{(p,p)}:=\sup\limits_{\|f\|_{p,\mu}\le
1}\|Af\|_{p,\mu}$$ is the operator norm of a linear operator $A$
on $L_{p,\mu}$, and $E_L(f)_{p,\mu}$ is the best approximation of
$f\in L_{p,\mu}$ from $\Pi_L^d$ defined by
$$E_L(f)_{p,\mu}:=\inf\limits_{P\in\Pi_L^d}\|f-P\|_{p,\mu}.$$
We note that property (c) is essential.

\begin{rem} Let $\eta$ be a filter, i.e., $\eta$ is a continuous function
satisfying $\chi_{\left[0,1\right
]}\leq\eta\leq\chi_{\left[0,2\right]}.$ We may weaken the
smoothness condition on $\eta$  such that the operator norms
$\|V_{L,\eta}\|$ are uniformly bounded. Wang and Sloan
investigated  the corresponding problem on the sphere,  and gave
the compact condition on $\eta$ for which the  operator norms of
the filtered polynomial operators $V^{\Bbb S}_{L,\eta}$ on the
sphere are uniformly bounded. Following the way in \cite{WS}, Li
obtained in \cite{Li} that  the operator norms  $\|V_{L,\eta}\|$
are uniformly bounded whenever $\eta\in W^{\lfloor
\frac{d+2\mu+1}{2}\rfloor}BV$, where $W^rBV[a,b]$ denotes the set
of all  continuous functions $\eta$ on $[a,b]$ for which
$\eta^{(r-1)}$ is absolutely continuous and $\eta^{(r)}_+$ and
$\eta^{(r)}_-$ exist and are of bounded variation on $[a,b]$ for $
r\in\Bbb N$. Hence, if $\eta\in W^{\lfloor
\frac{d+2\mu+1}{2}\rfloor}BV$, then properties (a)-(e) hold.
 Note that
$$C^{r+1}[a,b]\subset W^rBV[a,b]\subset C[a,b],\ {\rm for\ any}\
r\in\Bbb N.$$ The condition $\eta\in W^{\lfloor
\frac{d+2\mu+1}{2}\rfloor}BV$ is compact, since there exists an
$\eta\in W^{\lfloor \frac{d+2\mu-1}{2}\rfloor}BV$ such that
$\|V_{L,\eta}\|$ are not uniformly bounded.
\end{rem}

Now we define weighted  Besov spaces on the ball. Given $1\le p\le
\infty$, $r>0$, and $0<\tau\le\infty$, we define the weighted
Besov space $B_\tau^r(L_{p,\mu})$ to be the space of all real
functions $f$ with quasi-norm
$$\|f\|_{B_\tau^r(L_{p,\mu})}:=\Bigg\{
\begin{aligned}&\|f\|_{p,\mu}+\Big(\sum\limits_{j=0}^\infty2^{jr\tau}E_{2^j}(f)_{p,\mu}^\tau\Big)^{1/\tau},\ \ &0<\tau<\infty,
\\
&\|f\|_{p,\mu}+\sup\limits_{j\ge0}2^{jr}E_{2^j}(f)_{p,\mu},\ \ \
&\tau=\infty,\end{aligned}$$ while the weighted Besov class
$BB_\tau^r(L_{p,\mu})$ is defined to be the unit ball of the
weighted Besov space $B_\tau^r(L_{p,\mu})$.

There are other definitions of the weighted  Besov spaces which
are equivalent (see \cite[Proposition 5.7]{KPX}). We remark  that
if $1\le p,q\le\infty,\ r>(d+2\mu)/p$, then $B_\tau^r(L_{p,\mu})$
is compactly embedded into $C(\Bbb B^d)$, and for $1\le
p\le\infty,\ r>0$, $0<\tau_1\le\tau_2\le \infty$,
\begin{equation*}
B_{\tau_1}^r(L_{p,\mu})\subset B_{\tau_2}^r(L_{p,\mu})\subset
B_\infty^r(L_{p,\mu}).
\end{equation*}
It follows from  the Jackson inequality (see \cite{Xu}) that for
 $f\in W_{p,\mu}^{r},\ 1\le p\le \infty$, $r>0$,
\begin{equation}\label{2.1}
  E_{n}(f)_{p,\mu}\lesssim n^{-r}\|f\|_{W_{p,\mu}^{r}} .
\end{equation}
This means that
 $$W_{p,\mu}^r\subset B_\infty^r(L_{p,\mu}).$$
It can be seen that for $f\in B_\tau^r(L_{p,\mu})$, $0<\tau\le
\infty$,
\begin{equation}\label{2.1-0}
  E_{n}(f)_{p,\mu}\le 2^r n^{-r}\|f\|_{B_\tau^r(L_{p,\mu})}.
\end{equation}

We introduce a  metric $\rho$ on $\mathbb{B}^d$:
$$ \rho(x,y):= \arccos \Big( (x,y)+\sqrt{1-|x|^2}\sqrt{1-|y|^2}\Big).$$
For $r>0,\ x\in \mathbb{B}^d$ and a positive integer $n$,  we set
$${\bf B}_\rho(x,r):=\{y\in \mathbb{B}^d\ |  \ \rho(x,y)\le r\}.$$
For $\varepsilon\in(0,1)$, we say that a finite subset
$\Lambda\subset\mathbb{B}^d$ is maximal $\varepsilon$-separated if
$$\mathbb{B}^d\subset\bigcup\limits_{\oz\in\Lambda}{\bf B}_\rho(\oz,\varepsilon) \ {\rm\ and}\ \min\limits_{\oz\neq \oz^\prime}\rho(\oz,\oz^\prime)\ge \varepsilon.$$
Note that such a maximal $\varepsilon$-separated set $\Lambda$
exists and $\#\Lz\asymp \va^{-d}$, where $\# A$ denotes the number
of elements of a set $A$ (see \cite[Lemma 5.2]{PX}).

Finally, we give the Nikolskii inequalities on $\Bbb B^d$.

\begin{lem}\label{lem2.6-0}(\cite[Proposition 2.4]{KPX}) Let $ 1\le p,q\le \infty$ and $\mu\ge 0$. Then for  any $P\in \Pi_n^d$ we
have,
\begin{equation}\label{2.19}\|P\|_{q,\mu}\lesssim n^{(d+2\mu)(1/p-1/q)_+}\|P\|_{p,\mu}.\end{equation}\end{lem}

\section{Filtered  hyperinterpolation on the ball}

Let $\eta$ be a filter such that properties (a)-(e) hold. We want
to approximate  the inner product integral \eqref{3.1} of
$V_{L,\eta}(f)(x)$ by a positive quadrature  rule of polynomial
degree $3L$. Following \cite{SWo}, we shall call the resulting
operator  ``filtered hyperinterpolation".

For this purpose, we need  positive quadrature rules on $\Bbb
B^d$. For $L\in\Bbb N$, we assume that $Q_L(f):=
\sum_{\oz\in\Lz_L} \lz_\oz f(\oz)$ is a positive quadrature  rule
on $\Bbb B^d$ which is exact for $f\in \Pi_{3L}^d$, i.e., $\Lz_L$
is a finite subset of $\mathbb{B}^d$ with $\#\Lz_L\asymp L^d$,
weights $\lz_\oz>0,\ \oz\in \Lz_L$, satisfy, for all
$P\in\Pi_{3L}^d$,
$$\int _{\mathbb{B}^d}P(x)w_{\mu }(x)dx=
Q_{L}(P)=\sum_{\omega\in\Lambda _{L}}\lambda_{\omega}P(\omega).$$
Such positive quadrature rules exist. Indeed, for $L\in \Bbb N$,
it follows from \cite[Theorem 11.6.5]{DX} that for a given maximal
$\delta/L$-separated subset $\Lambda_L$ of $\Bbb B^d$ with
$\delta\in(0,\delta_0)$ for some $\delta_0>0$, there exists a
positive quadrature formula
\begin{equation*} \int _{\mathbb{B}^d}f(x)w_{\mu }(x)dx \approx Q_{L}(f):=\sum_{\omega\in\Lambda _{L}}\lambda_{\omega}f(\omega),\ \lambda _{\omega }>0,\end{equation*}
with $\lambda_{\omega}\asymp \int_{{\bf
B}_\rho(y,\delta/L)}w_\mu(x)dx$ and $\#\Lambda_L\asymp L^d$, which
is exact for $\Pi^d_{3L}$.

For the above positive quadrature formula $Q_L$, we can define the
discreted inner product $\langle \cdot,\cdot\rangle_{Q_L}$ on
$C(\Bbb B^d)$ by
\[\langle f,g\rangle_{Q_L} := Q_L(fg)=\sum _{\omega \in \Lambda_{L}}\lambda _{\omega }f(\omega )g(\oz).\]
The filtered hyperinterpolation operator is defined by
\begin{equation}\label{3.5}
G_{L}(f)(x):=\langle f, K_{L,\eta}(x,\cdot) \rangle_{Q_L}=\sum
_{\omega\in\Lambda_{L}}\lambda _{\omega}f(\oz)K_{L,\eta}(x,\oz).
\end{equation}
From \cite{WHLW} we know that
\begin{equation}\label{2.10}\|G_{L}\|=\sup_{x\in\mathbb{B}^d}\sum _{\omega\in\Lambda_{L}}\lambda
_{\omega}|K_{L,\eta}(x,{\omega})|\lesssim1.\end{equation}

The following theorem plays an important role in the proof of
upper estimates.

\begin{thm}\label{thm3}Let $1\leq p,q \leq \infty$, $r>(d+2\mu)/p$, $0<\tau\le\infty$, and $G_{L}$
   be given as in \eqref{3.5}. Then for all $ f\in X_p^r$, we have
\begin{equation}\label{3.8} \| f-G_{L}(f)\|_{q,\mu}\lesssim L^{-r+(d+2\mu)(\frac 1p-\frac1q)_+}\|f\|_{X_p^r},\end{equation}
where $X_p^r$ denotes $W_{p,\mu}^{r}$ or $B_\tau^r(L_{p,\mu})$.
\end{thm}

\begin{rem}When $1\le p=q\le \infty$ and $X_p^r=W_{p,\mu}^r$, \eqref{3.8} was obtained by Li in \cite{Li}. \end{rem}

In order to prove Theorem \ref{thm3}, we need the following two
lemmas.

\begin{lem}\label{lem3.4}(\cite[Theorem 3.1]{WHLW} and \cite[Lemma 2.2]{LW})
Suppose that $n\in\Bbb N$, $\Omega_n$ is a finite subset of
$\mathbb{B}^d$, and $\{\mu_{\omega
}:\,\omega\in\Omega_n\}$ is a set of positive numbers. If there exists a $p_{0}\in (0,\infty )$ such that
for any $f\in \Pi _{n}^{d}$,
\begin{equation}\label{2.6}
  \sum _{\omega \in \Omega_n}\mu_{\omega }| f(\omega )|^{p_{0}}\lesssim
\int_{\mathbb{B}^d}| f(x)|^{p_{0}}w_{\mu }(x)dx,
\end{equation}
then the following regularity
condition
\begin{equation}\label{2.7} \sum _{\omega \in
\Omega_n\cap {\bf B}_\rho(y,\frac{1}{n})}\mu_{\omega }
 \lesssim \int_{{\bf B}_\rho(y,\frac{1}{n})}w_\mu(x)dx,\ {\rm for\ any}\ y\in\Bbb B^d,\end{equation}holds.

Conversely, if the regularity condition \eqref{2.7} holds,
then for any $1\le p<\infty$, $m\in\Bbb N$, $m\geq n$, $f\in \Pi _{m}^{d}$, we have
\begin{equation}\label{2.8}\sum _{\omega \in \Omega_n}\mu_{\omega}\left|f(\omega)\right |^{p}\lesssim
(\frac{m}{n})^{d+2\mu}\int_{\mathbb{B}^d}\left| f(x) \right
|^{p}w_{\mu }(x)dx.\end{equation}
\end{lem}

\begin{lem}\label{lem2.6} Let $1\le p\le\infty$, and $L\in\Bbb N$. Then for any $N\ge L$ and $P\in\Pi_N^d$, we have
\begin{equation}\label{3.19} \|G_{L}(P)\|_{p,\mu}\lesssim(\frac{N}{L})^{\frac{d+2\mu}{p}}\|P\|_{p,\mu}.
\end{equation}
\end{lem}

\begin{proof}
Our proof will be divided into three cases.

\noindent\emph{Case 1:} $p=\infty$.

In this case, by \eqref{2.10} we have for $N\geq L$ and $P\in
\Pi_N^d$,
\begin{align*}
  \|G_{L}(P)\|_{\infty}\le \|G_L\|
 \|P\|_{\infty}\lesssim \|P\|_{\infty}.
\end{align*}

\noindent\emph{Case 2:} $p=1$.

In this case, since $Q_L$ is a positive quadrature  rule which is
exact for $\Pi_{3L}^d$, then \eqref{2.6} is true for
$\{\lambda_\omega\}_{\omega\in\Lambda_L}$ with $p_0=2$. This leads
that the regular condition \eqref{2.7} holds. By property (c) and
Lemma \ref{lem3.4} we obtain for $N\geq L$ and $P\in \Pi_N^d$,
\begin{align*}
 \|G_{L}(P)\|_{1,\mu}&= \| \sum\limits
_{\omega \in \Lambda _{L}}\lambda _{\omega }P({\omega
})K_{L,\eta }(\cdot,{\omega })\|_{1,\mu}\\&\leq \sum\limits _{\omega \in
\Lambda _{L}}\lambda _{\omega } |P({\omega })| \|
K_{L,\eta } (\cdot,{\omega }) \|_{1,\mu}\\&\lesssim \sum\limits _{\omega \in
\Lambda _{L}}\lambda _{\omega } |P({\omega })|\lesssim (\frac{N}{L})^{{d+2\mu}}\|P\|_{1,\mu}.
\end{align*}

\noindent\emph{Case 3:} $1<p< \infty$.

In this case, for $N\geq L$ and $P\in \Pi_N^d$, by the H\"{o}lder inequality,
\eqref{2.10}, property (c), and Lemma \ref{lem3.4}, we obtain
\begin{align*}
\|G_{L}(P)\|_{p,\mu}^{p}&=\int
_{\mathbb{B}^d}\Big|\sum \limits_{\omega\in\Lambda_{L}}\lambda
_{\omega}P({\omega})K_{L,\eta}(x,\oz)
\Big|^{p}w_{\mu}(x)dx\\&\leq \int _{\mathbb{B}^d}\Big( \sum
\limits_{\omega \in \Lambda_{L}}\lz_{\oz}|P(\omega )
|| K_{L,\eta}(x,{\omega})|\Big)^{p}w_{\mu }(x)dx
\\&\leq\int _{\mathbb{B}^d}
\Big(\sum\limits_{\omega \in \Lambda _{L}}\lambda_{\omega }
|P({\omega})|^{p}|K_{L,\eta }(x,{\omega })|\Big )\Big(\sum\limits _{\omega \in \Lambda _{L}}
\lambda_{\omega }|K_{L,\eta}(x,{\omega})|\Big)^{p-1}w_{\mu
}(x)dx
\\&\le
\Big(\sum\limits _{\omega \in \Lambda _{L}}\lz_{\oz}|P({\omega })
|^{p} \| K_{L,\eta }(\cdot, {\omega })
\|_{1,\mu}\Big)\Big(\sup_{x\in\mathbb{B}^d}\sum\limits _{\omega
\in \Lambda _{L}} \lambda_{\omega } |K_{L,\eta }(x,{\omega
})|\Big)^{{p-1}}
\\&\lesssim \sum\limits_{\omega \in \Lambda
_{L}}\lambda _{\omega } |P({\omega })|^{p}
\lesssim (\frac{N}{L})^{d+2\mu}\|P\|_{p,\mu}^p,
\end{align*}
which proves \eqref{3.19}.

Lemma \ref{lem2.6} is proved.\end{proof}

Now we turn to prove Theorem \ref{thm3}.

\vskip 2mm

\noindent{\it Proof of Theorem \ref{thm3}.}

Since  $X_{p}^r$ is continuously embedded into
$B_\infty^r(L_{p,\mu})$, it suffices to show Theorem \ref{thm3}
for  $B_\infty^r(L_{p,\mu})$.

Suppose that $m$ is an integer satisfying $2^{m}\leq L< 2^{m+1}$.
Let $1\le p,q\le \infty,\ r> (d+2\mu)/p$,
 and $f\in B_\infty^r(L_{p,\mu})$. Define
$g_{2^k}\in\Pi_{2^k}^d$ by
$$E_{2^k}(f)_{p,\mu}=\|f-g_{2^k}\|_{p,\mu},$$ and let $f_k=g_{2^k}-g_{2^{k-1}}$ for $k\ge0$, where we set $g_{2^{-1}}=0$. Note that since  $\Pi_{2^k}^d$
are the finite dimensional linear spaces, the best approximant
 polynomials $g_{2^{k}}$  always exist (see, for instance,
\cite[p. 17, Theorem 1]{L}). It can be seen that
$f_k\in\Pi_{2^k}^d$, and the series $\sum\limits_{k=j+1}^\infty
f_k$ converges to $f-g_{2^j}$ in the uniform norm for each
$j\ge-1$. We have
\begin{align}\label{3.18-0}
  \|f-G_{L}(f)\|_{q,\mu}&=
 \|\sum _{k=m+1}^{\infty }(f_k-G_{L}(f_k))\|_{q,\mu}\nonumber\\&\leq\sum_{k=m+1}^{\infty }
 \|f_k-G_{L}(f_k)\|_{q,\mu}\nonumber
 \\&\leq
\sum_{k=m+1}^{\infty }\|f_k\|_{q,\mu}+ \sum_{k=m+1}^{\infty }\|
G_{L}(f_k)\|_{q,\mu}.
\end{align}
It follows from  the Nikolskii
inequality \eqref{2.19} and \eqref{2.1-0} we have
\begin{align}\label{3.10}
\|f_k\|_{q,\mu}&\lesssim 2^{k(d+2\mu)(\frac 1p-\frac1q)_+}\|f_k\|_{p,\mu}\nonumber\\&
\le2^{k(d+2\mu)(\frac 1p-\frac1q)_+}(\|f-g_{2^k}\|_{p,\mu}+\|f-g_{2^{k-1}}\|_{p,\mu})\notag\\& \lesssim 2^{k(d+2\mu)(\frac 1p-\frac1q)_+}E_{2^{k-1}}(f)_{p,\mu}\nonumber
\\&\lesssim 2^{-k(r-(d+2\mu)(\frac 1p-\frac1q)_+)}\|f\|_{B_\infty^r(L_{p,\mu})} .
\end{align}
This means that
\begin{align}\label{3.11-1}
\sum_{k=m+1}^\infty \|f_k\|_{q,\mu}&\lesssim \sum_{k=m+1}^\infty
2^{-k(r-(d+2\mu)(\frac
1p-\frac1q)_+)}\|f\|_{B_\infty^r(L_{p,\mu})} \notag\\ &\asymp
2^{-m(r-(d+2\mu)(\frac
1p-\frac1q)_+)}\|f\|_{B_\infty^r(L_{p,\mu})}\notag\\ &\asymp
L^{-r+(d+2\mu)(\frac 1p-\frac1q)_+}\|f\|_{B_\infty^r(L_{p,\mu})}.
\end{align}
Applying \eqref{3.10} and  Lemma \ref{lem2.6}, we get for $k\ge
m+1$,
\begin{align*} \|G_{L}(f_k)\|_{q,\mu}&\lesssim L^{(d+2\mu)(\frac 1p-\frac1q)_+}\|G_{L}(f_k)\|_{p,\mu}\notag\\ &\lesssim  L^{(d+2\mu)(\frac 1p-\frac1q)_+}
(\frac{2^{k}}{L})^{\frac{d+2\mu}{p}}
 \|f_k\|_{p,\mu}\\ &\lesssim  L^{(d+2\mu)(\frac 1p-\frac1q)_+}
(\frac{2^{k}}{L})^{\frac{d+2\mu}{p}}
2^{-kr}\|f\|_{B_\infty^r(L_{p,\mu})}.
\end{align*}
It follows that
\begin{align}\label{3.19-0}
\sum_{k=m+1}^\infty \|G_{L}(f_k)\|_{q,\mu}&\lesssim
\sum_{k=m+1}^\infty  L^{(d+2\mu)(\frac 1p-\frac1q)_+}
(\frac{2^{k}}{L})^{\frac{d+2\mu}{p}}2^{-kr} \|f\|_{B_\infty^r(L_{p,\mu})} \notag\\
&\lesssim L^{(d+2\mu)(\frac
1p-\frac1q)_+-\frac{d+2\mu}p}\sum_{k=m+1}^\infty
2^{-k(r-\frac{d+2\mu}p)}\|f\|_{B_\infty^r(L_{p,\mu})}\notag\\
&\asymp L^{-r+(d+2\mu)(\frac
1p-\frac1q)_+}\|f\|_{B_\infty^r(L_{p,\mu})}.
\end{align}
Hence, for  $f\in B_\infty^r(L_{p,\mu})$, by \eqref{3.18-0},
\eqref{3.11-1}, and \eqref{3.19-0} we have
\begin{align*}\|f-G_{L}(f)\|_{q,\mu}
\lesssim L^{-r+(d+2\mu)(\frac 1p-\frac1q)_+}\|
f\|_{B_\infty^r(L_{p,\mu})}.\end{align*}

This completes the proof of Theorem \ref{thm3}. $\hfill\Box$

\

Next we show that the filtered hyperinterpolation operators  $G_L$
are order optimal in sense of the optimal recovery in some cases.
Let $F_d$ be a class of continuous functions on $D_d$, and
$(X,\|\cdot\|_X)$ be a normed linear space of functions on $D_d$,
where $D_d$ is a subset of $\Bbb R^d$. For $n\in \Bbb N$, the
sampling numbers (or the optimal recovery) of $F_d$ in $X$ are
defined by
$$g_n(F_d,X) \ :=\inf_{\sub{\xi_1,\dots,\xi_n\in D_d\\ \varphi:\ \Bbb R^n\rightarrow X}}
\sup_{f\in F_d}\|f-\varphi(f(\xi_1),\dots,f(\xi_n))\|_X,
$$where the infimum is taken over all $n$ points
$\xi_1,\dots,\xi_n$ in $D_d$ and all  mappings $\varphi$ from
$\Bbb R^n$ to $X$. If in the above definition, the infimum is
taken over all $n$ points $\xi_1,\dots,\xi_n$ in $D_d$ and all
linear mappings $\varphi$ from $\Bbb R^n$ to $X$, we obtain  the
definition of the linear sampling numbers $g_n^{\rm lin}(F_d,X)$.

It is well known (see \cite{TWW}) that for a balanced convex set
$F_d$,
\begin{equation}\label{3.16}
  g_n^{\rm lin}(F_d,X)\ge g_n(F_d,X)\ge\inf_{\sub{\xi_1,\dots,\xi_n\in D_d}}\sup_{\sub{f\in F_d\\
  f(\xi_1)=\dots=f(\xi_n)=0}}\|f\|_X.
\end{equation}

\begin{thm} \label{thm3.5}Let $1\le q\le p\le\infty$, $0<\tau\le\infty$, and
$r>(d+2\mu)/p$. Then we have
\begin{equation}\label{3.16-0}
  g_n^{\rm lin}(BX_p^{r},L_{q,\mu })
  \asymp g_n(BX_p^{r},L_{q,\mu})\asymp n^{-r/d},
\end{equation}where $X_p^r$ denotes $W_{p,\mu}^{r}$ or $B_\tau^r(L_{p,\mu})$.

\end{thm}

In order to prove Theorem \ref{thm3.5} we need the following
lemma.

\begin{lem}\label{lem4.1}(\cite[Proposition 4.8]{DW})
Let $X$ be a linear subspace of \ $\Pi_N^d$ with ${\rm
dim}\,X\ge\varepsilon\,{\rm dim}\,\Pi_N^d$ for some
$\varepsilon\in(0,1)$. Then there exists a function $f\in X$ such
that $\|f\|_{p,\mu}\asymp1$ for all $0<p\le\infty$.
\end{lem}

\

\noindent{\it Proof of Theorem \ref{thm3.5}.}

Without loss of generality we assume that $n$ is sufficiently
large. We choose $L\in\Bbb N$ such that
$$\#\Lz_L\le n\ \ {\rm and} \ \ \#\Lz_L\asymp n.$$ It follows from the definition of $G_L$ and $g_n^{\rm
lin}(BX_p^{r},L_{q,\mu })$ that $$ g_n^{\rm
lin}(BX_p^r,L_{q,\mu})\le\sup_{f\in BX_p^r}\|f-G_L(f)\|_{q,\mu}.$$
By Theorem \ref{thm3} and the above inequlity we have
\begin{align}\label{111-1}g_n(BX_p^r,L_{q,\mu})\le g_n^{\rm lin}(BX_p^r,L_{q,\mu})\lesssim L^{-r}\asymp n^{-r/d}.\end{align}

Now we show the lower bound. Let $\xi_1,\dots,\xi_n$ be any $n$
distinct points on $\Bbb B^d$. Take a positive integer $N$ such
that $2n\le{\rm dim}\,\Pi_N^d\le Cn$, and denote
$$X_0:=\big\{g\in\Pi_N^d\ \big|\ g(\xi_j)=0\ {\rm for\ all}\ j=1,\dots,n\big\}.$$
Thus, $X_0$ is a linear subspace of $\Pi_N^d$ with $${\rm
dim}\,X_0\ge{\rm dim}\,\Pi_N^d-n\ge\frac12{\rm dim}\,\Pi_N^d.$$ It
follows from Lemma \ref{lem4.1} that there exists a function
$g_0\in X_0$ such that $$\|g_0\|_{p_0,\mu}\asymp1,\ {\rm}\ {\rm
for\ all}\ 0<p_0\le\infty.$$ Let $f_0(x)=N^{-r}(g_0(x))^2$ and
$m\in\Bbb N$ such that $2^{m-1}\le N<2^m$. Then by the fact that
$E_{2^j}(f_0)_{p,\mu}\le \|f_0\|_{p,\mu}$ we have
\begin{align*}
  \|f_0\|_{B_\tau^r(L_{p,\mu})} &=\|f_0\|_{p,\mu}+\Big(\sum\limits_{j=0}^{m+1}2^{jr\tau} E_{2^j}(f_0)_{p,\mu}^\tau\Big)^{1/\tau}
  \\&\lesssim \Big(\sum\limits_{j=0}^{m+1}2^{jr\tau} \Big)^{1/\tau}\|f_0\|_{p,\mu}\lesssim 2^{mr}N^{-r}\|g_0\|_{2p,\mu}^2\lesssim
  1.
\end{align*} By the fact that $f_0\in \Pi_{2N}^d$ and  \eqref{122}  we  have
\begin{align}\label{2.35}
  \|f_0\|_{W_{p,\mu}^r} &=\|f_0\|_{p,\mu}+\|(-D_\mu)^{r/2}f_0\|_{p,\mu}\lesssim N^r\|f_0\|_{p,\mu}=\|g_0\|_{2p,\mu}^2\lesssim
  1.
\end{align}
Hence, there exists a positive constant $C$ such that $f_1=Cf_0\in
BX_p^r$, and $f_1(\xi_1)=\dots=f_1(\xi_n)=0$. It follows from
\eqref{3.16} that
\begin{equation*}\label{111-2}g_n(BX_p^r,L_{q,\mu})\ge\inf_{\sub{\xi_1,\dots,\xi_n\in
D_d}}\|f_1\|_{q,\mu}\gtrsim N^{-r}\inf_{\sub{\xi_1,\dots,\xi_n\in
D_d}}\|g_0\|_{2q,\mu}^2\asymp N^{-r}\asymp
n^{-r/d},\end{equation*} which combining with \eqref{111-1}, gives
\eqref{3.16-0}.

The proof of Theorem \ref{thm3.5} is finished.
$\hfill\Box$

\begin{rem} It follows from the proof of  Theorem \ref{thm3.5} that  for $1\le q\le p\le\infty$, $0<\tau\le\infty$, and $r>(d+2\mu)/p$,
$$g_n(BX_p^r
,L_{q,\mu})  \asymp n^{-r/d} \asymp \sup_{f\in
BX_p^r}\|f-G_L(f)\|_{q,\mu},$$where $X_p^r$ denotes
$W_{p,\mu}^{r}$ or $B_\tau^r(L_{p,\mu})$. This implies that the
filtered hyperinterpolation operators  are asymptotically optimal
algorithms in the sense of optimal recovery for $1\le q\le
p\le\infty$.
\end{rem}

Finally  we prove Theorem \ref{thm1}.

\

\noindent{\it Proof of  Theorem \ref{thm1}.}

Since $\|f\|_{B_\infty^r(L_{p,\mu})}\lesssim \|f\|_{W_{p,\mu}^r} $
for $f\in W_{p,\mu}^r$,   by \eqref{1.3} we obtain for $1\le
p\le\infty$ and $r>(d+2\mu)/p$,  $$e_{n}^{\rm
det}(BW^r_{p,\mu})\lesssim e_{n}^{\rm
det}(BB_\infty^r(L_{p,\mu}))\lesssim n^{-r/d}.
$$

 Now we prove the lower bound. Let $\xi_1,\dots,\xi_n$ be any $n$
distinct points on $\Bbb B^d$. Take a positive integer $N$ such
that $2n\le{\rm dim}\,\Pi_N^d\le Cn$. According to the proof of
Theorem \ref{thm3.5}, there exists a function $f_1(x)$ such that
$$f_1\in BW_{p,\mu}^r,\ \ f_1(\xi_1)=\dots=f_1(\xi_n)=0, \ \
f_1(x)\ge 0,$$ and $$\|f_1\|_{1,\mu}=\int_{\Bbb
B^d}f_1(x)w_\mu(x)dx\asymp N^{-r}.$$ It follows from \cite{TWW}
that for $BW_{p,\mu}^r$ which is a balanced convex set,
\begin{align*}
  e_n^{\rm det}(BW_{p,\mu}^r)&\ge\inf_{\sub{\xi_1,\dots,\xi_n\in \Bbb B^d}}\sup_{\sub{f\in F_d\\
  f(\xi_1)=\dots=f(\xi_n)=0}} \Big|\int_{\Bbb B^d}f(x)w_\mu(x)dx\Big|\\ &\ge
  \inf_{\sub{\xi_1,\dots,\xi_n\in \Bbb B^d}}\Big|\int_{\Bbb B^d}f_1(x)w_\mu(x)dx\Big|\gtrsim N^{-r}\asymp n^{-r/d}.
\end{align*}

The proof of Theorem \ref{thm1} is finished. $\hfill\Box$

\section{The upper estimates}\label{sec3}

 This section is devoted to  proving  the upper
estimates of the quantities $e_n^{\rm ran}(X_p^r)$  given as in
\eqref{1.5}. That is, for $1\le p\le\infty$, $r>(d+2\mu)/p$, and
$0<\tau\le\infty$,
\begin{equation}\label{4.1}e_{n}^{\rm ran}(BX_p^r)\lesssim
n^{-\frac{r}{d}-\frac{1}{2}+(\frac{1}{p}-\frac{1}{2})_+},
\end{equation}where $X_p^r$ denotes
$W_{p,\mu}^{r}$ or $B_\tau^r(L_{p,\mu})$.

For this purpose, we will use the positive quadrature rule and the
filtered hyperinterpolation operator to construct an randomized
algorithm to attain the upper bounds. Due to Henrich \cite{H4}, we
need a concrete Monte Carlo method by virtue of the standard Monte
Carlo algorithm. It is defined as follows: let $\{\xi_i\}_{i=1}^N$
be independent, $\Bbb B^{d}$-valued, distributed over $\Bbb B^d$
with respect to the measure $w_\mu(x)dx$ random vectors on
probability space $(\Omega,\Sigma,\nu)$. For any $h\in C(\Bbb
B^d)$, we put
\begin{equation}\label{96-0} Q_N^\oz(h)=\frac1N\sum_{i=1}^N
h(\xi_i(\oz)),\ \oz\in\Omega.
\end{equation} The
following lemma  can be drawn in a same way as in
\cite[Proposition 5.4]{H4}. Here we omit the proof.
\begin{lem}\label{prop4.2} Let $1\le p\le\infty$ and $\mu\ge0$,  Then for any $h\in C(\Bbb B^d)$, we have
\begin{equation}\label{97-0}\Bbb E_\oz |{\rm INT}_d(h)-Q_N^\oz(h)|\lesssim N^{-\frac12+(\frac1p-\frac12)_+}\|h\|_{p,\mu}.
\end{equation}
\end{lem}

\vskip1mm

Now we prove \eqref{4.1}.

\vskip2mm

\noindent{\it Proof of \eqref{4.1}.}

Let $n\in\Bbb N$. Without loss of generality we assume that $n$ is
sufficiently large. Then there exists a  positive quadrature rule
\begin{equation*} \int _{\mathbb{B}^d}f(x)w_{\mu }(x)dx \approx
Q_{L}(f):=\sum_{\xi\in\Lambda _{L}}\lambda_{\xi}f(\xi),\ \lambda
_{\xi }>0,\end{equation*} which is exact for $\Pi_{3L}^d$, where
$\# \Lz_L\le n/2$ and $n\asymp L^d$.

As in Section 3.1, we can construct the filtered
hyperinterpolation operator by
$$
G_{L}(f)(x)=\sum _{\xi \in \Lambda _{L}}\lambda_{\xi
}f(\xi)K_{L,\eta }(x,\xi),
$$where $K_{L,\eta}$ is given as in \eqref{98-0}.
Hence, according to Theorem \ref{thm3} we have for any $f\in
X_p^r,\ 1\leq p\leq \infty,\ r> (d+2\mu)/{p}$,
\begin{equation}\label{4.3-0}\|f-G_{L}(f)\|_{p,\mu}\lesssim L^{-r}\|f\|_{X_p^r}. \end{equation}

 Let $N=\lfloor \frac n2\rfloor$, where $\lfloor x\rfloor$ denotes the largest integer not exceeding $x$.
 We define the randomized algorithm $(A_n^{\omega})$   by
\begin{align*}A_n^{\omega}(f)=Q_N^\omega(f-G_{L}(f))+{\rm INT}_d(G_{L}(f)),
\end{align*}where  $f\in C(\Bbb B^d)$, and $Q_N^\oz$ is the standard Monte Carlo algorithm
given as in \eqref{96-0}.  We also note that $${\rm
INT}_d(G_{L}(f))=\sum_{\xi\in \Lz_\xi}\lz_\xi f(\xi).
$$Clearly, the algorithm $(A_n^{\omega})$ use only at most
$\#\Lz_L+N\le n$ function values of $f$. This means that
$(A_n^\oz)\in \mathcal A_n^{\rm ran}$. Also the algorithm
$(A_n^{\omega})$ is a randomized linear algorithm. It is easy to
check that
$$
|{\rm INT}_d(f)-A_n^{\rm \oz}(f)|=|{\rm INT}_d(g)-Q_N^\oz(g)|,
$$where $g=f-G_L(f)$.
Combining with \eqref{97-0} and \eqref{4.3-0}, we obtain for $f\in
BX_p^r$, $r>(d+2\mu)/p$,
\begin{align*}\Bbb E_\oz|{\rm INT}_d(f)-A_n^{\oz}(f)|&=\Bbb E_\oz |{\rm INT}_d(g)-Q_N^\oz(g)|\\&\lesssim N^{-\frac12+(\frac1p-\frac12)_+}\|f-G_L(f)\|_{p,\mu}
\\&\lesssim N^{-\frac12+(\frac1p-\frac12)_+}L^{-r}\asymp n^{-\frac rd-\frac12+(\frac1p-\frac12)_+},
\end{align*}which leads to
$$e_n^{\rm ran}(BX_p^r)\le e^{\rm ran}(BX_p^r, (A_n^{\omega}))\lesssim n^{-\frac rd-\frac12+(\frac1p-\frac12)_+}.$$

This completes  the proof of \eqref{4.1}.
 $\hfill\Box$

\section{ Lower estimates}\label{sec4}

This section is devoted to proving the lower estimates of the
quantities $e_n^{\rm ran}(X_p^r)$  given as in \eqref{1.5}. That
is, for $1\le p\le\infty$, $r>(d+2\mu)/p$, and $0<\tau\le\infty$,
\begin{equation}\label{4.3}e_{n}^{\rm ran}(BX_p^r)\gtrsim
n^{-\frac{r}{d}-\frac{1}{2}+(\frac{1}{p}-\frac{1}{2})_+},
\end{equation}where $X_p^r$ denotes
$W_{p,\mu}^{r}$ or $B_\tau^r(L_{p,\mu})$. Theorem \ref{thm2}
follows from \eqref{4.1} and \eqref{4.3} immediately.

The proof of \eqref{4.3} is based on the idea of Bakhvalov in
\cite{B1} and Novak in \cite{N2,N3}.

\begin{lem}\label{lem5.1} (See \cite[Lemma 3]{N2}.)

\noindent (a) Let $F\subset L_{1,\mu}$ and $\psi_j,\,j = 1, \dots, 4n$, with the
following conditions:

(i) the $\psi_j$ have disjoint supports and satisfy
$${\rm INT_d}(\psi_j)=\int_{\Bbb B^d} \psi_j(x)w_\mu(x)dx\ge \dz,\ {\rm for}\ j= 1,\dots, 4n.$$

 (ii) $F_1:=\big\{\sum\limits_{j=1}^{4n}
\az_j\psi_j\,\big|\,\alpha_j\in\{-1,1\}\big\}\subset F$.

\noindent Then
$$e_n^{\rm ran}(F)\ge  \frac{1}{2}\dz n^{\frac12}.$$

\noindent (b) We assume that instead of (ii) in statement (a) the
property \vskip 1mm

 (ii') $F_2:=\big\{\pm \psi_j\,\big|\, j= 1, \dots,
4n\big\}\subset F$. \vskip 1mm

\noindent Then
$$e^{\rm ran}_n (F)\ge \frac14\dz.$$
\end{lem}

By this lemma, we proceed to construct a sequence of functions
$\{\psi_j\}_{j=1}^{4n}$ satisfying the conditions of Lemma
\ref{lem5.1} for $F=BX_p^r$, where $X_p^r$ denotes $W_{p,\mu}^{r}$
or $B_\tau^r(L_{p,\mu})$,  $1\le p\le \infty,\ r>(d+2\mu)/p,\
0<\tau\le\infty$.

\

For a given $n\in \Bbb N$, choose a positive integer $m$
satisfying $n\asymp m^{d}$, $m\ge 6$, and
$\{x_j\}_{j=1}^{4n}\subset{\bf B}({\bf 0},\frac23)$ such that
$${\bf B}(x_i,\frac2m)\bigcap{\bf B}(x_j,\frac2m)=\emptyset,\ \
{\rm if}\ \ i\neq j,$$ where ${\bf B}(\xi,r)=\big\{x\in\Bbb
B^d\,\big|\,|x-\xi|\le r\big\}$ for $\xi\in\Bbb B^d$ and $r>0$.

Let $\varphi$ be a nonnegative
$C^\infty$-function on $\mathbb{R}^d$ supported in $\Bbb B^d$
and being equal to 1 on ${\bf B}({\bf 0},\frac12)$. We define
$$\varphi_j(x)=\varphi(m(x-x_j)),\  j=1,\dots,4n.$$
 Clearly,
$${\rm supp}\,\varphi_j\subset{\bf B}(x_j,\frac1m)\subset{\bf B}({\bf 0},\frac56),\  j=1,\dots,4n,$$and
\begin{equation}\label{5.0}{\rm supp}\,\varphi_i\cap {\rm supp}\,\varphi_j=\emptyset,\ {\rm for}\ i\neq
j.\end{equation} It is easy to verify that
\begin{equation}\label{5.1}\|\varphi_j\|_{p,\mu}\asymp\Big(\int_{{\bf B}(x_j,\frac1m)}|\varphi_{j}(x)|^pdx\Big)^{1/p}=\Big(\int_{{\bf B}({\bf 0},\frac1m)}|\varphi(mx)|^pdx\Big)^{1/p}\asymp  m^{-d/p}.
\end{equation}

We set $$F_0:=\big\{f_{\boldsymbol\az}:=\sum\limits_{j=1}^{4n}
\az_j\varphi_j\,\big|\,
{\boldsymbol\az}=(\az_1,\cdots,\az_{4n})\in\mathbb{R}^{4n}\big\}.$$
Then we have the following lemma.

\begin{lem}\label{lem44} If $f_{\boldsymbol\az}\in F_0$, then for $r>0$, $1\le p\le\infty$, and $0<\tau\le\infty$,
\begin{align}\label{5.31}\|f_{{\boldsymbol\az}}\|_{X_p^r}\lesssim m^{r-d/p}\|{\boldsymbol\az}\|_{l_p^{4n}},
\end{align}where $X_p^r$ denotes $W_{p,\mu}^{r}$ or $B_\tau^r(L_{p,\mu})$,
and $$\|{\boldsymbol\az}\|_{\ell_p^{4n}}:=\left\{
\begin{aligned}&\big(\sum_{j=1}^{4n}|\az_j|^p\big)^{1/p},\ \ &1\leq
p<\infty,
\\
&\max\limits_{1\le j\le
4n}|\az_j|,\ \ \
&p=\infty.\end{aligned}\right.$$

\end{lem}

\begin{proof}
 Indeed, for any $f_{\boldsymbol\az}\in F_0$,  it follows from \eqref{5.0} and \eqref{5.1} that
\begin{equation}\label{5.2}\|f_{\boldsymbol\az}\|_{p,\mu}\asymp
 m^{-d/p}\|{\boldsymbol\az}\|_{\ell_p^{4n}}. \end{equation}
For a positive integer $v>r$, by the definition of $-D_\mu$ and \eqref{5.0}, it is easy to verify that
$${\rm supp}\,(-D_\mu)^v\varphi_i\,\bigcap\, {\rm supp}\,(-D_\mu)^v\varphi_j=\emptyset,\ {\rm for}\ i\neq
j,$$and
$$\|(-D_\mu)^v\varphi_j\|_{p,\mu}\lesssim m^{2v-d/p},$$ which leads to
\begin{equation}\label{5.3-0}
  \|(-D_\mu)^v f_{\boldsymbol\az}\|_{p,\mu}\lesssim m^{2v-d/p}\|{\boldsymbol\az}\|_{\ell_p^{4n}}.
\end{equation}
It follows from the Kolmogorov type inequality (see \cite[Theorem 8.1]{Di}) that
\begin{equation}\label{5.3}
 \|(-D_\mu)^{r/2} f_{\boldsymbol\az}\|_{p,\mu}\lesssim\|(-D_\mu)^v f_{\boldsymbol\az} \|_{p,\mu}^{\frac{r}{2v}}\, \|f_{\boldsymbol\az}\|_{p,\mu}^{\frac{2v-r}{2v}}
\lesssim m^{r-d/p}\|{\boldsymbol\az}\|_{\ell_p^{4n}},
\end{equation}which combining with \eqref{5.2}, we obtain \eqref{5.31} for $W_{p,\mu}^r$.

By the fact that
$$\|f_{\boldsymbol\az}\|_{B_\infty^r(L_{p,\mu})}\lesssim
\|f_{\boldsymbol\az}\|_{B_\tau^r(L_{p,\mu})},\ 0<\tau<\infty,$$ it
suffices to show \eqref{5.31} for $B_\tau^r(L_{p,\mu})$,
$0<\tau<\infty$. Since $E_{2^j}(f_{\boldsymbol\az})_{p,\mu}\le
\|f_{\boldsymbol\az}\|_{p,\mu}$ for any $j\ge0$, by \eqref{5.31}
 we have
\begin{align}\label{5.12}
  \sum\limits_{2^j<m}\Big(2^{jr} E_{2^j}(f_{\boldsymbol\az})_{p,\mu}\Big)^\tau&\le \|f_{\boldsymbol\az}\|_{p,\mu}^\tau\sum\limits_{2^j<m}2^{jr\tau}\notag\\&\asymp m^{r\tau}\|f_{\boldsymbol\az}\|_{p,\mu}^\tau\asymp m^{(r-d/p)\tau}\|{\boldsymbol\az}\|_{\ell_p^{4n}}^\tau.
\end{align}
Choose a positive number $\upsilon>r$, by \eqref{2.1} we obtain for any $j\ge0$,
 $$E_{2^j}(f_{\boldsymbol\az})_{p,\mu}\lesssim 2^{-j\upsilon}\|f_{\boldsymbol\az}\|_{W_{p,\mu}^\upsilon},$$ which combining with \eqref{5.31} for $W_{p,\mu}^r$, we get
 \begin{align}\label{5.13}
  \sum\limits_{2^j\ge m}\Big(2^{jr} E_{2^j}(f_{\boldsymbol\az})_{p,\mu}\Big)^\tau&
  \le \|f_{\boldsymbol\az}\|_{W_{p,\mu}^\upsilon}^\tau\sum\limits_{2^j\ge m}2^{j(r-\upsilon)\tau}\nonumber
  \\&\asymp m^{(r-\upsilon)\tau}\|f_{\boldsymbol\az}\|_{W_{p,\mu}^\upsilon}^\tau\lesssim m^{(r-d/p)\tau}\|{\boldsymbol\az}\|_{\ell_p^{4n}}^\tau.
\end{align}
It follows from \eqref{5.2}, \eqref{5.12}, and \eqref{5.13} that
 \begin{align*}
 \|f\|_{B_\tau^r(L_{p,\mu})}&:=\|f\|_{p,\mu}+\Big( \sum\limits_{j=0}^\infty\Big(2^{jr} E_{2^j}(f_{\boldsymbol\az})_{p,\mu}\Big)^\tau\Big)^{1/\tau}\lesssim m^{r-d/p}\|{\boldsymbol\az}\|_{\ell_p^{4n}}.
\end{align*}

This completes the  proof of Lemma 5.2.
\end{proof}

\vskip2mm

Finally we turn to prove \eqref{4.3}.

\

\noindent{\it Proof of \eqref{4.3}.}

First we consider the case $2\le p\le \infty$. By the fact that
$$X_\infty^r\subset X_p^r,\ \  2\le p\le \infty,$$ it suffices to consider the case $p=\infty$.
It follows from \eqref{5.31} that  when $\az_j\in\{-1,1\}, \,
j=1,\dots, 4n,$
$$\|f_{\boldsymbol\az}\|_{X_\infty^r}\lesssim m^{r}\|{\boldsymbol\az}\|_{\ell_\infty^{4n}}\lesssim m^r.$$
Hence, there exists a positive constant $C_1$ such that
$C_1m^{-r}f_{\boldsymbol\az}\in BX_\infty^r$. Set
$$\psi_j(x):=C_1m^{-r}\varphi_j(x),\ j=1,\dots,4n.$$
We have
$$F_1:=\Big\{\sum\limits_{j=1}^{4n}
\az_j\psi_j\ \big| \ \alpha_j\in\{-1,1\},\
j=1,\dots,4n\Big\}\subset BX_\infty^r.$$ It follows from
\eqref{5.1} that
$${\rm INT}_d(\psi_j)=\int_{\Bbb B^d}\psi_j(x)w_\mu(x)dx=C_1 m^{-r}\|\varphi_j\|_{1,\mu}\asymp m^{-r-d},$$
Applying Lemma \ref{lem5.1} (a) we obtain for $2\le p\le
\infty$,
\begin{equation}\label{5.5}e_n^{\rm ran}(BX_p^r)\ge e_n^{\rm ran}(BX_\infty^r)\gtrsim m^{-r-d}n^{1/2}\asymp
n^{-\frac{r}{d}-\frac12}.\end{equation}

Next we consider the case $1\le p<2$. It follows from \eqref{5.31}
that $$\|\pm \varphi_j\|_{X_p^r}\lesssim m^{r-d/p}.$$Hence, there
exists a positive constant $C_2$ such that
$$\psi_j(x):=C_2m^{-r+d/p}\varphi_j(x)\in BX_p^r,\ j=1,\dots,4n.$$
We have$$F_2:=\{\pm\psi_j:\ j=1,\dots,4n\}\subset BX_p^r.$$ It
follows from \eqref{5.1} that
$${\rm INT}_d(\psi_j)=\int_{\Bbb B^d}\psi_j(x)w_\mu(x)dx=C_2 m^{-r+d/p}\|\varphi_j\|_{1,\mu}\asymp m^{-r+d/p-d}.$$
Applying Lemma \ref{lem5.1} (b), we obtain for $1\le p<2$,
\begin{equation*}e_n^{\rm ran}(BX_p^r)\gtrsim m^{-r+d/p-d}\asymp
n^{-\frac{r}{d}+\frac1p-1},\end{equation*} which combining with
\eqref{5.5}, gives  the lower bounds of $e_{n}^{\rm
ran}(BX_p^{r})$ for $1\le p\le\infty$.

This completes the proof of \eqref{4.3}.
 $\hfill\Box$

\

\noindent{\bf Acknowledgment}  Jiansong Li and Heping Wang  were
supported by the National Natural Science Foundation of China
(Project no. 11671271).


\begin{thebibliography}{99}

\bibitem{B1} N. S. Bakhvalov, On approximate computation of integrals, Vestnik MGV, Ser. Math. Mech. Aston. Phys. Chem. 4 (1959) 3-18.
\bibitem{B2} N. S. Bakhvalov, On a rate of convergence of indeterministic integration processes within the functional classes $W_{p}^{l}$,
Theory Probab. Appl. 7 (1962) 226-227.
\bibitem{B3} N. S. Bakhvalov, On the optimality of linear methods for operator approximation in convex
classes of functions, USSR Computational Mathematics and Mathematical Physics, 11 (1971) 244-249.
\bibitem{BH}  J. S. Brauchart, K. Hesse, Numerical integration over spheres of arbitrary dimension,  Constr. Approx. 25 (1) (2007) 41-71.
\bibitem{BV} V. A. Bykovskii, On the correct order of the error of optimal cubature formulas in spaces with dominant derivative,
and on quadratic deviations of grids. Akad. Sci. USSR, Vladivostok, Computing Center Far-Eastern Scientific Center, 1985.

\bibitem{DW} F. Dai, H. Wang, Optimal cubature formulas in weighted Besov spaces with $A_{\infty}$ weights on multivariate domains, Constr. Approx. 37 (2013) 167-194.
\bibitem{DX} F. Dai, Y. Xu, Approximation Theory and Harmonic Analysis on Spheres and Balls, Springer, 2013.
\bibitem{Di} Z. Ditzian, Fractional derivatives and best approximation, Acta Math. 81 (4) (1998) 323-348.
\bibitem{DY} L. Duan, P. Ye, Integration over the sphere $\mathbb{S}^{d-1}$ on Besov classes in different settings,
Preprint.
\bibitem{Du1} V. V. Dubinin, Cubature formulas for classes of functions with
bounded mixed difference. Math. Sb. 183 (7) (1992) 23-34.
\bibitem{Du2} V. V. Dubinin,  Cubature formulas for Besov classes, Izv. Ross.
Akad. Nauk Ser. Math. 61 (2) (1997) 27-52.
\bibitem{DU} D. Dung, T. Ullrich,  Lower bounds for the integration error
for multivariate functions with mixed smoothness and optimal
Fibonacci cubature for functions on the square, Math. Nachr. 288 (7) (2015)
 743-762.


\bibitem{FY1} G. Fang,  P. Ye, Complexity of deterministic and
randomized methods for multivariate integration problems for the
class $H^p(I_d)$, IMA J. Numer. Anal. 25 (3) (2005) 473-485.
\bibitem{FKK} K. K. Frolov, Upper bounds on the error of quadrature formulas on classes of functions, Doklady Akademy Nauk USSR, 231 (1976) 818-821.



\bibitem{H4} S. Heinrich, Random approximation in numerical analysis, Appl. Math. 150 (1994) 123-171.
\bibitem{HS} K. Hesse, I. H. Sloan, Hyperinterpolation on the
sphere, in: ``Frontiers in Interpolation and Approximation
(Dedicated to the memory of Ambikeshwar Sharma)" (editors: N.K.
Govil, H.N. Mhaskar, R.N. Mohapatra, Z. Nashed, J. Szabados),
Chapman \& Hall/CRC, 2006 pp. 213-248.
\bibitem{H} K. Hesse, A lower bound for the worse-case cubature error on sphere of arbitrary dimension,  Numer. Math. 103 (2006) 413-433.
\bibitem{HMS} K. Hesse, H. N. Mhaskar, I. H. Sloan, Quadrature in Besov spaces on the Euclidean sphere, J. Complexity 23 (2007) 528-552.


\bibitem{KN} D. Krieg, E. Novak, A universal algorithm for multivariate integration, Found. Comput. Math. 17 (4) (2017) 895-916.
\bibitem{KPX} G. Kyriazis, P. Petrushev, Y. Xu, Decomposition of weighted
Triebel-Lizorkin and Besov spaces on the ball, Proc. London Math.
Soc. 97 (2008) 477-513.


\bibitem{LW} Jiansong Li, Heping Wang, Weighted $\ell_q$ approximation problems on the ball and on the sphere, preprint.

\bibitem{Li} Jingyu Li, On Filtered Polynomial Approximation on the Unit
Ball, Master Thesis, Capital Normal University,  2018.
\bibitem{L} G.G. Lorentz, Approximation of Functions. Holt, Rinehart and Winston, New York, 1966.

\bibitem{N2} E. Novak, Stochastic properties of quadrature formulas, Numer. Math. 53 (1988) 609-620.

\bibitem{N3} E. Novak, Deterministic and Stochastic Error Bound in Numerical Analysis, Springer, Berlin, 1988.

\bibitem{N4} E. Novak, Some results on the complexity of numerical
integration. Monte Carlo and quasi-Monte Carlo methods,
Springer Proc. Math. Stat. 163 (2016) 161-183.

\bibitem{NUU} V. K. Nguyen, M. Ullrich, T. Ullrich, Change of variable in spaces of mixed smoothness and numerical integration of multivariate
functions on the unit cube, Constr. Approx. 46 (2017) 69-108.

\bibitem{PX} P. Petrushev, Y. Xu, Localized polynomial frames on the ball,
Constr. Approx. 27 (2008) 121-148.

 \bibitem{R1} M. Reimer, Hyperinterpolation on the sphere at the minimal
projection order, J. Approx. Theory 104 (2) (2000) 272-286.

\bibitem{SWo} I. H. Sloan, R. S. Womersley, Filtered hyperinterpolation: a constructive polynomial approximation on the sphere, GEM Int. J. Geomath. 3 (1) (2012) 95-117.

\bibitem{SM} M. M. Skriganov, Constructions of uniform distributions in terms of geometry of numbers, St. Petersburg Math. J. 6 (1995) 635-664.


\bibitem{T3} V. N. Temlyakov, On a way of obtaining lower estimates for the error of quadrature formulas, Math. USSR Sb. Russian, 181 (1990) 1403-1413.
 English translation: Math. USSR Sbornik, 71 (1992) 247-257.
\bibitem{TWW} J. F. Traub, G. W. Wasilkowski, H. Wo\'zniakowski, Information-Based Complexity,
Academic Press, New York, 1988.



\bibitem{UU} M. Ullrich, T. Ullrich, The role of Frolov's cubature formula for functions with bounded mixed derivative, SIAM J. Numer. Anal. 54 (2) (2016) 969-993.

\bibitem{U} M. Ullrich,  A Monte Carlo method for integration of multivariate smooth functions, SIAM J. Numer. Anal. 55 (3) (2017)  1188-1200.

\bibitem{W1} H. Wang, Quadrature formulas for classes of functions with
Bounded mixed derivative or difference. Science in China (Series
A) 40 (5) (1997) 429-495.

\bibitem{Wang} H. Wang, Optimal lower estimates for the worst case quadrature
error and the approximation by hyperinterpolation operators in the
Sobolev space setting on the sphere, Int. J. Wavelets
Multiresolut. Inf. Process. 7 (6) (2009) 813-823.


\bibitem{WS} H. Wang, I. H. Sloan, On filtered polynomial approximation on the sphere, J. Fourier Anal. Appl. 23 (4) (2017) 863-876.

\bibitem{WZ} H. Wang, Y. Zhang, Optimal randmized quadrature for Sobolev classes on the sphere, preprint.

\bibitem{WHLW} H. Wang, Z. Huang, C. Li, L. Wei, On the norm of the
hyperinterpolation operator on the unit ball, J. Approx. Theory
192 (2015) 132-143.

\bibitem {Xu1} Y. Xu, Summability of Fourier orthogonal series for
Jacobi weight on a ball in $\Bbb R^d$, Trans. Amer. Math. Soc. 351
(1999) 2439-2458.
\bibitem {Xu} Y. Xu, Weighted approximation of functions on the unit sphere, Constr. Approx. 21 (2005) 1-28.

\bibitem{Y} P. Ye, Computational complexity of the integration problem for
anisotropic classes,  Adv. Comput. Math. 23 (4) (2005)  375392.
\end{thebibliography}
\end{document}